\documentclass[11pt]{article}
\usepackage{amssymb,amsthm,amsmath,mathrsfs}
\usepackage{prettyref,cite}
\usepackage{hyperref}
\usepackage{graphicx}
\usepackage [a4paper, footskip=1cm, headheight = 16pt, top=2cm, bottom=2.5cm,  right=2cm,  left=2cm]{geometry}
\newcommand\pref[1]{\prettyref{#1}}
\newrefformat{chp}{\chaptername~\ref{#1}}
\newrefformat{sec}{Section~\ref{#1}}
\newrefformat{subsec}{Subsection~\ref{#1}}
\newrefformat{thm}{Theorem~\ref{#1}}
\newrefformat{lem}{Lemma~\ref{#1}}
\newrefformat{cor}{Corollary~\ref{#1}}
\newrefformat{con}{Conjecture~\ref{#1}}
\newrefformat{ques}{Question~\ref{#1}}
\newrefformat{prop}{Proposition~\ref{#1}}
\newrefformat{dfn}{Definition~\ref{#1}}
\newrefformat{exm}{Example~\ref{#1}}
\newrefformat{rem}{Remark~\ref{#1}}
\newrefformat{fig}{\figurename~\ref{#1}}
\newrefformat{tbl}{\tablename~\ref{#1}}
\newrefformat{eq}{$($\ref{#1}$)$}
\newrefformat{fac}{Fact~\ref{#1}}

\newtheorem{pro}{Proposition}[section]

\newtheorem{lem}[pro]{Lemma}
\newtheorem{thm}[pro]{Theorem}
\newtheorem{con}[pro]{Conjecture}
\newtheorem{ques}[pro]{Question}
\newtheorem{predef}[pro]{Definition}
\newtheorem{prerem}[pro]{Remark}
\newenvironment{rem}{\begin{prerem}\rm}{\hfill $\blacktriangle$\end{prerem}}

\newcommand{\bq}{\begin{equation}}
\newcommand{\eq}{\end{equation}}

\begin{document}
\title{On a question of Erd\H{o}s and Faudree on the size Ramsey numbers}
\author{\small  Ramin Javadi$^{\textrm{a},1}$,  Gholamreza Omidi$^{\textrm{a},\textrm{b},2}$ \\
	\small  $^{\textrm{a}}$Department of Mathematical Sciences,
	Isfahan University
	of Technology,\\ \small P.O.Box: 84156-83111, Isfahan, Iran\\
	\small  $^{\textrm{b}}$School of Mathematics, Institute for
	Research
	in Fundamental Sciences (IPM),\\
	\small  P.O.Box: 19395-5746, Tehran, Iran\\
	\small Emails: \texttt{\href{mailto:rjavadi@cc.iut.ac.ir}{rjavadi@cc.iut.ac.ir},
		\href{mailto:romidi@cc.iut.ac.ir}{romidi@cc.iut.ac.ir}}
	}
\date{}
\maketitle
\footnotetext[1] {\tt Corresponding author.}

\footnotetext[2] {\tt This research is partially
	carried out in the IPM-Isfahan Branch and in part supported
	by a grant from IPM (No. 95050217).}
\begin{abstract}
For given simple graphs $G_1$ and $G_2$, the size Ramsey number $\hat{R}(G_1,G_2)$ is the smallest positive integer $m$, where there exists a graph $G$ with $m$ edges such that in any edge coloring of $G$ with two colors red and blue, there is either a red copy of $G_1$ or a blue copy of $G_2$. In 1981, Erd\H{o}s and Faudree investigated the size Ramsey number $\hat{R}(K_n,tK_2)$, where $K_n$ is a complete graph on $n$ vertices and $tK_2$ is a matching of size $t$. They obtained the value of $\hat{R}(K_n,tK_2)$ when $n\geq 4t-1$ as well as for $t=2$ and asked for the behavior of these numbers when $ t $ is much larger than $ n $. In this regard, they posed the following interesting question: For every positive integer $n$, is it true that 
$$\lim_{t\to \infty} \frac{\hat{R}(K_n,tK_2)}{t\, \hat{R}(K_n,K_2)}=  \min\left\{\dfrac{\binom{n+2t-2}{2}}{t\binom{n}{2}}\mid t\in \mathbb{N}\right\}?$$
In this paper, we obtain the exact value of $ \hat{R}(K_n,tK_2) $ for every positive integers $ n,t $ and as a byproduct, we give an affirmative answer to the question of Erd\H{o}s and Faudree.\\

\textbf{Keywords:} Ramsey numbers, size Ramsey numbers, stripes.\\
\textbf{AMS subject classification:} 05C55, 05D10.
\end{abstract}

\section{Introduction}\label{sec:intro}

In this paper, all graphs are finite, undirected and 
simple. We also follow \cite{Boundy} for the terminology
and notation not defined here. Let $G,G_1,\ldots,G_n$ be given graphs. We write $G\rightarrow (G_1,\ldots,G_n)$, if in every edge coloring of $G$ with $n$ colors,  there is a monochromatic copy of $G_i$ of $i$-th color for some $1\leq i\leq n$. The minimum number of ``vertices'' of a graph $G$ such that $G\rightarrow (G_1,\ldots,G_n)$ is called the \textit{Ramsey number} of $G_1,\ldots,G_n$ and is denoted by $R(G_1,\ldots,G_n)$. The Ramsey numbers have been studied widely in the literature. Another important parameter is the \textit{size Ramsey number} which is defined as the minimum number of ``edges'' of a graph $G$ such that $G\rightarrow (G_1,\ldots,G_n)$ and is denoted by $ \hat{R}(G_1,\ldots,G_n)$. One may easily observe that $$ \hat{R}(G_1,\ldots,G_n)\leq \binom{R(G_1,\ldots,G_n)}{2}.$$ The study of Ramsey numbers and also size Ramsey numbers is an important task in Ramsey theory and has been at the center of attention for the last three decades, e.g. see \cite{CFS,FS,RS} and  references therein.

In this paper, we focus on the size Ramsey number $\hat{R}(G,H)$ when one of the graphs $G,H$ is a matching. The investigation of the size Ramsey number of a graph paired with a matching was initiated by Erd\H{o}s and Faudree \cite{EF}. They gave some bounds, asymptotic results and in some cases the exact values for $ \hat{R}(G,tK_2) $, for several classical graphs $G$ such as paths, cycles, complete graphs and complete bipartite graphs. In particular, when $ G $ is the complete graph $ K_n $, it is known that the Ramsey number $ R(K_n,tK_2)$ is equal to $n+2t-2 $  (see \cite{L,FS-matching}) and thus, $ \hat{R}(K_n,tK_2)\leq \binom{n+2t-2}{2} $. Erd\H{o}s and Faudree in \cite{EF} proved that the equality holds when $ n $ is large with respect to $ t $. More precisely, they showed the following.
\begin{thm}\label{thm:erdos1} {\rm \cite{EF}}
	For every integers $ t\geq 1 $ and $ n\geq 4t-1 $, we have
	\begin{equation} \label{eq:Kn}
	\hat{R}(K_n,tK_2)= \binom{R(K_n,tK_2)}{2}= \binom{n+2t-2}{2}.
	\end{equation}
\end{thm}
It should be noted that the condition $ n\geq 4t-1 $ in \pref{thm:erdos1} is not tight. However, for small values of $n$, they showed that the statement of this theorem does not hold, as follows.
\begin{thm} \label{thm:erdos2} {\rm \cite{EF}}
	For every positive integer $ n $, we have
	\[\hat{R}(K_n,2K_2)= \begin{cases}
	2\binom{n}{2} & 2\leq n\leq 5, \\
	\binom{n+2}{2} & n\geq 6.
	\end{cases}
	\]
\end{thm}
The case when $ t $ is much larger with respect to $ n $ is more interesting, where the equality in \eqref{eq:Kn} does not hold.
In this regard, for every graph $ G $, Erd\H{o}s and Faudree \cite{EF} defined
\[ \hat{r}_{_\infty} (G)=\lim_{t\to \infty} \frac{\hat{R}(G,tK_2)}{t\, \hat{R}(G,K_2)}.
\]
They proved the limit always exists. In particular, they showed that $ 2/n\leq r_\infty(K_n)\leq 8/n $. Also, it is evident that $$\frac{\hat{R}(K_n,tK_2)}{t\, \hat{R}(K_n,K_2)} \leq \frac{\binom{n+2t-2}{2}}{t\binom{n}{2}}.$$
Thus, 
\begin{equation} \label{eq:Mn}
r_\infty (K_n) \leq \min\left\{\dfrac{\binom{n+2t-2}{2}}{t\binom{n}{2}}\mid t\in \mathbb{N}\right\}.
\end{equation}
This gives rise to the following question due to Erd\H{o}s and Faudree.

\begin{ques} \label{ques:erdos} {\rm \cite{EF}}
	For every positive integer $ n $, define
	\[M_n= \min\left\{\dfrac{\binom{n+2t-2}{2}}{t\binom{n}{2}}\mid t\in \mathbb{N}\right\}.\]	
	Is it true that $  \hat{r}_{_\infty}(K_n)=M_n $?
\end{ques}

They proved that the answer to \pref{ques:erdos} is positive  when $ n\leq 4 $. In this paper, we determine the exact value of $\hat{R}(K_n,tK_2)$ for every positive integers $n,t$ and through this, we give an affirmative answer to \pref{ques:erdos}. \\

The paper is organized as follows. In Section~\ref{sec:main}, for given graphs $G$ and $H$ (with no isolated vertex), we give a lower bound for the size Ramsey number $\hat{R}(G,H) $ in terms of the chromatic number of $G$ and the number of vertices of $H$ (see \pref{thm:main}). Then, we use this result to compute the exact value of $\hat{R}(K_n,tK_2)$ for every positive integers $n,t$ (see \pref{thm:kng2}).
In Section 3, we provide some implications of \pref{thm:kng2}. In particular, we improve \pref{thm:erdos1} and find a necessary and sufficient condition for positive integers $ n,t $ for which \eqref{eq:Kn} holds (see \pref{thm:nlarge}). We also give an affirmative answer to \pref{ques:erdos} for every positive integer $n$ (see \pref{thm:EFQ}). Finally in Section $4$, we present some further remarks and an open problem.

\section{Main results} \label{sec:main}
Let $ G $ be a given graph. Also, suppose that the graph $ H $ is the disjoint union of the graphs $ H_1,\ldots, H_l $ and for every $ 1\leq i\leq l $ assume that $F_i$ is a graph with minimum $|E(F_i)|$ such that  $ F_i\rightarrow (G,H_i) $. Now, let $F$ be the disjoint union of $ F_1,\ldots, F_l $. It is clear that $ F \rightarrow (G,H) $ and so we have $\hat{R}(G,H)\leq |E(F)|=\sum_{i=1}^{l} |E(F_i)|.$ On the other hand, for every $ 1\leq i\leq l$ we have $$|E(F_i)|=\hat{R}(G,H_i)\leq \binom{R(G,H_i)}{2}.$$  Therefore, we have
\begin{equation}
\label{eq:union}
\hat{R}(G,H) \leq \min \left\{
\sum_{i=1}^{l} \binom{R(G,H_i)}{2} : H \text{ is the disjoint union of } H_1,\ldots, H_l
\right\}.
\end{equation}


Moreover, it is known that for every positive integers $ n,t $,  $ R(K_n,tK_2)= n+2t-2 $ (see \cite{L}).
Now, for every positive integers $ n,t $, define
\[
g(n,t)= \min \left\{ \sum_{i=1}^l \binom{n+2s_i-2}{2} : l\in \mathbb{Z}^+, s_i\in \mathbb{Z}^+, s_1+\cdots+s_l\geq t
\right\}.
\]
 Therefore, \eqref{eq:union} ensures that $ \hat{R}(K_n,tK_2)\leq g(n,t) $. The main goal of this paper is to prove the other side of this inequality, thereby determining the exact value of
$ \hat{R}(K_n,tK_2)$. For some technical reasons, we need to define the following similar function. For every positive integers $ n,t $, define
\[
\hat{g}(n,t)= \min \left\{ \sum_{i=1}^l \binom{n+s_i-2}{2} : l\in \mathbb{Z}^+, s_i\in \mathbb{Z}^+, s_1+\cdots+s_l\geq 2t
\right\}.
\]
It is evident that the functions $ g(n,t) $ and $ \hat{g}(n,t) $ are both increasing with respect to $n$ and $ t $.
The following theorem provides a lower bound for the size Ramsey number of graphs in terms of the functions $ \hat{g}(n,t) $ and $ g(n,t)$.
A vertex coloring of a given graph $G$ is called {\it proper} if any two adjacent vertices receive different colors. The minimum number of colors for which $ G $ has a proper vertex coloring is called the {\it chromatic number} of $ G $ and is denoted by $\chi(G)$.

\begin{thm} \label{thm:main}
Let $ n\geq 2,t\geq 1 $ be two integers. Also, let $ G $ be a graph such that $ \chi(G)\geq n $ and $ H $ be a graph with no isolated vertex such that $ |V(H)|\geq 2t $. Then, $\hat{R}(G,H)\geq \hat{g}(n,t)$. Moreover, if $ H $ contains $ tK_2 $, then $ \hat{R}(G,H)\geq g(n,t)  $.
\end{thm}

The above theorem immediately implies that the exact value of $ \hat{R}(K_n,tK_2) $ is equal to $ g(n,t) $ for all integers $ n,t $.
In fact, the inequality $ \hat{R}(K_n,tK_2)\leq g(n,t) $ follows from \eqref{eq:union} and the inequality $ \hat{R}(K_n,tK_2)\geq g(n,t)  $ follows from  \pref{thm:main}. Therefore, we conclude the following theorem, which is the main result of the paper.
\begin{thm} \label{thm:kng2}
For every integers $ n\geq 2,t\geq 1 $, we have $\hat{R}(K_n,tK_2)= g(n,t)  $.
\end{thm}

In order to prove Theorem~\ref{thm:main}, we need the following two technical lemmas. The first lemma shows that the difference of the functions $ g(n,t) $ and $ \hat{g}(n,t) $ is bounded by a function of $ n $ when $ n\geq 4 $.
\begin{lem} \label{lem:diff}
For every integers $ n\geq 4 $ and $ t\geq 1 $, we have
\[ 0\leq g(n,t)-\hat{g}(n,t)\leq \frac{n-3}{2}. \]
\end{lem}
\begin{proof}
The left inequality immediately follows from the definitions. In order to prove the right inequality, let $I=(s_1,\ldots, s_l)  $ be a sequence of positive integers such that $$ s_1\geq s_2\geq \cdots\geq s_l,~~ s_1+\cdots+s_l\geq 2t,~~ \sum_{i=1}^{l} \binom{n+s_i-2}{2} =\hat{g}(n,t).$$ 
Then, we have $ s_1-s_l\leq 1 $. To see this, note that if $ s_1 \geq s_l+2 $, then in the sequence $I$ one can replace $ s_1,s_l $ with $ s_1-1, s_l+1 $, respectively  to obtain the new sequence $\hat{I}=(\hat{s}_1,\ldots,\hat{s}_l)$ with $\hat{s}_1+\cdots+\hat{s}_l\geq 2t$ and
$$\sum_{i=1}^{l} \binom{n+\hat{s}_i-2}{2}< \sum_{i=1}^{l} \binom{n+s_i-2}{2} =\hat{g}(n,t),$$
which is in contradiction with the definition of $ \hat{g}(n,t) $. Also, we have $s_1+\cdots+s_l=2t $, since otherwise if $ s_1+\cdots+s_l>2t $, then one can replace $ s_1 $ with $ s_1-1 $ in the sequence $I$ to decrease $ \sum_{i} \binom{n+s_i-2}{2} $ once more, which is impossible. Therefore, all $ s_i $'s are among at most two consecutive integers. Let $ x $ be an odd integer such that $ s_i\in\{x-1,x,x+1\} $, for all $ 1\leq i\leq l $ and let $ 2q $ be the number of occurrence of $ x $ in $ I $ (i.e. the number of odd integers in $ I $). Furthermore, choose $ I $ such that $ 2q $ is minimal subject to the conditions  $ s_1+\cdots+s_l=2t $ and $ \sum_{i} \binom{n+s_i-2}{2} =\hat{g}(n,t) $.

Now, we are going to prove that $ 2q\leq n-3 $. For the contrary, suppose that $ 2q\geq n-2 $. Now, if $ x\leq n-3 $, then let $I'= ( s'_1,\ldots, s'_{l-1}) $ be a sequence of integers which is obtained from $ I $ by removing one number $ x $ and replacing $ x $ numbers $ x $ with $ x+1 $. Then, $ s'_1+\cdots+s'_{l-1}=2t $ and
\begin{align*}
  \sum_{i=1}^{l-1} \binom{n+s'_i-2}{2} &=  \sum_{i=1}^l \binom{n+s_i-2}{2}- \binom{n+x-2}{2} + x(n+x-2)  \\
  & = \hat{g}(n,t)+ (n+x-2) (\frac{x-n+3}{2}) \leq \hat{g}(n,t).
\end{align*}
Thus, the definition of $\hat{g}(n,t)$ implies that $x=n-3$ and $\sum_{i=1}^{l-1} \binom{n+s'_i-2}{2}=\hat{g}(n,t).$ Nevertheless, the number of odd integers in $ I' $ is less than $ 2q $ which is in contradiction with the minimality of $ 2q $.
Also, if $ x\geq n-2 $, then let $I''= ( s''_1,\ldots, s''_{l+1}) $ be a sequence of integers which is obtained from $ I $ by adding one number $ n-2 $ and replacing $ n-2 $ numbers $ x $ with $ x-1 $. Then, $ s''_1+\cdots+s''_{l+1}=2t $ and
\begin{align*}
\sum_{i=1}^{l+1} \binom{n+s''_i-2}{2} &=  \sum_{i=1}^l \binom{n+s_i-2}{2} + \binom{2n-4}{2}-(n-2)(n+x-3) \leq \hat{g}(n,t).
\end{align*}
Again, we have $\sum_{i=1}^{l+1} \binom{n+s''_i-2}{2}=\hat{g}(n,t)$. Moreover, the number of odd integers in $ I'' $ is at most $ 2q-(n-3) < 2q $ which is in contradiction with the minimality of $ 2q $.
This implies that $ 2q\leq n-3 $. Now, let $(s'''_1,\ldots,s'''_l)$ be the sequence of even integers which is obtained from $ I $ by replacing $ 2q $ numbers $ x $ in $ I $ with $ q $ numbers $ x-1 $ and $ q $ numbers $ x+1 $. Then, $ s'''_1+\cdots+s'''_l=2t $ and every number $ s'''_i $ is even. Hence,
\[g(n,t)\leq \sum_{i=1}^{l} \binom{n+s'''_i-2}{2} = \sum_{i=1}^{l} \binom{n+s_i-2}{2}+q \leq \hat{g}(n,t)+\frac{n-3}{2}. \]
This completes the proof.
\end{proof}

It is noteworthy that Lemma~\ref{lem:diff} is not true for $ n\leq 3 $. For instance, for $ n=3 $, we have $ \hat{g}(3,t)=2t $ and $ g(3,t)=3t $.

We also need the following lemma which is interesting in its own right. It shows that for every integers $ n\geq 3,t\geq 1 $, if $ G $ is a graph with not many edges (depending on $ n $ and $ t $), then one may remove at most $ 2t-1 $ (or $ 2t $) vertices to diminish the chromatic number of $ G $ to less than $ n-2 $. Note that for a subset of vertices $ S\subseteq V(G) $, by $ G[S] $ we mean the induced subgraph of $G$ on  $ S $.
\begin{lem} \label{lem:chi}
	Let $ n\geq 3,t\geq 1 $ be two integers and $ G $ be a given graph.
\begin{itemize}
\item[{\rm (i)}] If $ |E(G)|<\hat{g}(n,t) $, then there exists a subset $ S\subseteq V(G) $ such that $ |S|\leq 2t-1 $ and $ \chi(G-S)\leq n-2 $.
\item [{\rm (ii)}] If $ |E(G)|<g(n,t) $, then there exists a subset $ S\subseteq V(G) $ such that $ \chi(G-S)\leq n-2 $ and $ G[S] $ does not contain $ tK_2 $. Also, if $ n\geq 4 $, then $ |S|\leq 2t $.
\end{itemize}
\end{lem}
\begin{proof}

First let $ n=3 $. In this case, let $ S $ be a minimum vertex cover for $ G $ (a set of vertices with minimum size containing at least one endpoint of each edge of $G$). The set $ V(G)\setminus S $ is a stable set (a set of pairwise nonadjacent vertices) in $ G $  and thus, $ \chi(G-S)\leq 1=n-2 $. If $|E(G)|<\hat{g}(3,t)=2t$, then $ |S|\leq |E(G)|\leq 2t-1 $ and thus, the case (i) is proved.  Now, suppose that $ |E(G)|< g(3,t)=3t $.  Then, by the minimality of $S$, every vertex in $ S $ has a neighbor in $ V(G)\setminus S $ and thus $ G[S] $ does not contain $ tK_2 $ (otherwise, $|S|\geq 2t$ and we have $ |E(G)|\geq 2t+t=3t $). This proves (ii) for $ n=3 $. So, henceforth suppose that $ n\geq 4 $.

Let $ \chi(G)=f $. If $ f\leq n-2$, then setting $ S=\emptyset $, we are done. So, suppose that $ f\geq n-1 $.
Now, among all proper vertex colorings of $G$ with $ f $ colors, consider the coloring with the color classes $ C_1,\ldots, C_{f} $ which maximizes $ \sum_{i} |C_i|^2 $. Also, without loss of generality, assume that $$|C_1|\geq |C_2|\geq \cdots \geq |C_{f}|.$$ Then, we have the following claim.\\

\noindent
\textbf{Claim 1.} For every $i<j$, every vertex in $ C_j $ has a neighbor in $ C_i $. \\

\noindent
For the contrary, assume that there is a vertex $ v\in C_j $ with no neighbor in $ C_i $. Then, transfer $ v $ from $ C_j $ to $ C_i $ to obtain a new proper coloring with color classes $ C'_1,\ldots, C'_f $, where
\[\sum_{k} |C'_k|^2=\sum_{k} |C_k|^2+2+ 2|C_i|-2|C_j|>\sum_{k} |C_k|^2, \]
a contradiction with the choice of the coloring. This proves Claim~1.\\

Now, let $S=\cup_{i=n-1}^{f} C_i $. Obviously, $ \chi(G-S)=n-2 $. Thus, if $ |S|\leq 2t-1 $, then we are done. So, assume that $ |S|\geq 2t $.
Let $ l=|C_{n-1}| $ and for each $ 1\leq i\leq l $, define $ f_i $ to be the maximum number $ k $ for which $ |C_{k}|\geq i $. Also, for each $ 1\leq i\leq l $, define $ s_i=f_i-(n-2) $. Then, $ |S|=s_1+\cdots+s_l\geq 2t $.
Now, define
\[N(f_1,\ldots, f_l) = l(1+\cdots+(f_l-1))+ (l-1) (f_l+\cdots+(f_{l-1}-1)) + \cdots+ (f_2+\cdots+(f_1-1)). \]
Thus,
\[ N(f_1,\ldots, f_l) = \sum_{i=1}^l \binom{f_i}{2} = \sum_{i=1}^l \binom{n+s_i-2}{2} \geq \hat{g}(n,t). \]
	
Also, by Claim~1, we have $|E(C_i,\cup_{j<i} C_j)|\geq |C_i|(i-1)$, where by $E(A,B)$ we mean the set of all edges between $A$ and $B$. On the other hand,

\begin{equation*}
N(f_1,\ldots, f_l)= l(1+\cdots+(f_l-1))+\sum_{i=f_l+1}^{f}|C_i|(i-1)\leq \sum_{i=2}^{f}|C_i|(i-1).
\end{equation*}

Therefore,
$$|E(G)|=\sum_{i=2}^{f}|E(C_i,\cup_{j<i} C_j)|\geq \sum_{i=2}^{f}|C_i|(i-1)\geq N(f_1,\ldots,f_l) \geq \hat{g}(n,t). $$ This proves (i).
	
Now, in order to prove (ii), suppose that $ |E(G)|< g(n,t) $. Moreover, note that if $ |E(G)|< \hat{g}(n,t) $, then the result follows from (i). So, suppose that $ |E(G)| = \hat{g}(n,t)+\varepsilon $, where $ \varepsilon\geq 0 $. By \pref{lem:diff}, we have $ 0\leq \varepsilon < (n-3)/2 $.
First, suppose that $ |S|= s_1+\cdots+s_l>2t $. Then,
\begin{align*}
|E(G)|& \geq  N(f_1,\ldots,f_l) = \sum_{i=1}^l \binom{n+s_i-2}{2} \\ & \geq \sum_{i= 2}^l \binom{n+s_i-2}{2}+\binom{n+s_1-3}{2}+ {n-2} \geq \hat{g}(n,t)+{n-2}>  \hat{g}(n,t)+\varepsilon,
\end{align*}
which is in contradiction with our assumption $|E(G)| = \hat{g}(n,t)+\varepsilon$. Thus, $ |S|= s_1+\cdots+s_l=2t $. In the sequel, we are going to prove that $ G[S] $ does not contain $ tK_2 $. For this, first we prove some facts, as follows. \\

\textbf{Claim 2.} For every $ (n-1)/2\leq i\leq n-1  $, we have $ |C_i|= l $ and there is some $ (n-1)/2\leq i_0\leq n-1  $ such that every vertex in $ \cup_{j=n}^f C_j $ has a unique neighbor in $ C_{i_0} $. \\

If $ |C_h|\geq l+1 $, for some $ (n-1)/2\leq h\leq n-1 $, then by Claim~1, we have
\begin{align*}
|E(G)|&=\sum_{i=2}^{f}|E(C_i,\cup_{j<i} C_j)| \geq \sum_{i=2}^{f}|C_i|(i-1) \\
&\geq (l+1)(1+\cdots+(h-1))+l(h+\cdots+(f_l-1))+\sum_{i=f_l+1}^{f}|C_i|(i-1) \\
& = N(f_1,\ldots, f_l)+\binom{h}{2} \geq \hat{g}(n,t)+\frac{n-3}{2}\geq g(n,t),
\end{align*}
a contradiction. To see the second fact, for the contrary, suppose that for every $ (n-1)/2\leq i\leq n-1  $, there is a vertex in $ \cup_{j=n}^f C_j $ with at least two neighbors in $ C_i $. Then,  by Claim~1, we have
\begin{align*}
|E(G)|&=\sum_{i=2}^{f}|E(C_i,\cup_{j<i} C_j)| \geq \sum_{i=2}^{f}|C_i|(i-1)+ \frac{n-2}{2} \\
& \geq N(f_1,\ldots, f_l)+ \frac{n-2}{2}\geq \hat{g}(n,t)+\frac{n-2}{2},
\end{align*}
again, a contradiction. This proves Claim~2.

Now, due to Claim~2, there is some $ (n-1)/2\leq i_0\leq n-1  $ such that every vertex in $ \cup_{j=n}^f C_j $ has a unique neighbor in $ C_{i_0} $. Also, all sets $ C_{\lceil (n-1)/2\rceil }, \ldots, C_{n-1} $ are of the same size and so are interchangeable. Thus, without loss of generality, we assume that $ i_0=n-1 $, i.e. every vertex in $ \cup_{j=n}^f C_j $ has a unique neighbor in $ C_{n-1} $. \\

\noindent
\textbf{Claim 3.} Every vertex in $ C_{n-1} $ has at most one neighbor in $ C_i $, for every $ n\leq i\leq f $. \\

Let $ x $ be a vertex in $ C_{n-1} $ and $ X $ be the set of all neighbors of $ x $ in $ C_i $, for some $ n\leq i\leq f $. By the above assumption, every member of $ X $ has no neighbor in $ C_{n-1}\setminus \{x\} $.  Now, if $ |X|\geq 2 $, then move $ x $ from $ C_{n-1} $ to $ C_i $ and move all members of $ X $ from $ C_i $ to $ C_{n-1} $. This increases the summation $ \sum_{k} |C_k|^2 $, a contradiction with the choice of the coloring. Thus, $ |X|\leq 1 $ and this proves Claim 3. \\

\noindent
\textbf{Claim 4.} Let $ x\in C_i $, for some $ n+1\leq i\leq f $ and $ y $ be the unique neighbor of $ x $ in $ C_{n-1} $. Then, for every $ n\leq j \leq i-1 $, the vertices $ x $ and $ y $ have a unique common neighbor in $ C_j $. \\

Let $ j $ be an integer where $ n\leq j \leq i-1 $. By Claim~3, $ x $ and $ y $ have at most one common neighbor in $ C_j $. Now, suppose that $ x $ and $ y $ have no common neighbor in $ C_j $ and let $ X $ be the set of neighbors of $ x $ in $ C_j $. By the above assumption, every vertex in $ X $ has a unique neighbor in $ C_{n-1} \setminus \{y\} $ and by Claim~3, these neighbors are distinct. Let $ Y $ be the set of all neighbors of $ X $ in $ C_{n-1} $. Thus, $ |Y|=|X| $ and $ x $ has no neighbor in $ Y $.  Also, by Claim~3, every member of $ Y $ has no neighbor in $ C_j\setminus X $. Now, move all members of $ X $ from $ C_j $ to $ C_{n-1} $, move all members of $ Y $ from $ C_{n-1} $ to $ C_j $ and move $ x $ from $ C_i $ to $ C_j $. This gives a new coloring for $ G $ which increases the summation $ \sum_{k} |C_k|^2 $, a contradiction with the choice of the coloring. This proves Claim~4. \\

\textbf{Claim 5.} The graph $G[S]$ (the induced subgraph of $ G $ on $S$) has the union of $l$ vertex-disjoint cliques of sizes $ s_1,\ldots, s_l $ as a spanning subgraph.\\

Let $ C_f=\{x_1,\ldots,x_k \} $ be the last color class. Then, $ f_1=\cdots=f_k=f  $. Now, every vertex $ x_i $ has a unique neighbor $ y_i $ in $ C_{n-1} $. By Claim~3, the vertices $ y_1,\ldots, y_k $ are distinct. Also, by Claim~4, the vertices $ x_i$ and $y_i $ have a unique common neighbor $ z^j_{i} $ in $ C_j $, for every $ n\leq j\leq f-1 $. Define, $ K_i=\{z_i^n, \ldots, z_i^{f-1}\}\cup \{x_i,y_i\} $, $ 1\leq i\leq k $. Then, $ K_i $ is a clique of $ G $. To see this, on the contrary, suppose that $ z_i^{j} $ is nonadjacent to $ z_i^{j'} $, for some $ n\leq j<j'\leq f-1 $. Thus, by Claim~4, $ z_i^{j'} $ and $ y_i $ have a common neighbor in $ C_j $ and thus $ y_i $ has two neighbors in $ C_j $, which is in contradiction with Claim~3. Hence, $ K_i $ is a clique of $ G $.  Also, $ |K_1|=\cdots=|K_k|=s_1=\cdots=s_k $ and the vertices of the cliques $ K_1,\ldots, K_{k} $ are disjoint, since otherwise there is a vertex in $ \cup_{i=n}^f C_i $ with two neighbors in $ C_{n-1} $ or a vertex in $C_{n-1}$ with two neighbors in $C_j$ for some $n\leq j\leq f$. Now, we can remove the vertices of the cliques $ K_1,\ldots, K_k  $ from $ G $ and continue with  the color class $ C_{f_{k+1}} $ with similar arguments. This proves Claim~5. \\

Now, let $O=(r_1,\ldots, r_{2q})$ be the sequence of all odd integers in the sequence $ I=(s_1,\ldots, s_l)$, where $ r_1\geq r_2\geq \cdots\geq r_{2q} $. Also, let $I'=(s'_1,\ldots, s'_l) $ be the sequence of even integers formed by concatenation of the sequences $I\setminus O $ and $J=(r_1-1,\ldots,r_q-1,r_{q+1}+1,\ldots, r_{2q}+1)$, i.e. $I'$ is the sequence obtained from $I$ by replacing the elements in $O$ with the elements in $J$. Then, clearly $ s'_1+\cdots+s'_l=s_1+\cdots+s_l=2t $ and we have
\[|E(G)|< g(n,t)\leq \sum_{i=1}^{l}\binom{n-2+s'_i}{2}\leq \sum_{i=1}^{l}\binom{n-2+s_i}{2}+q = N(f_1,\ldots, f_l)+q. \]
Also, by Claim~5, $G[S]$ contains the union of $l$ vertex-disjoint cliques of sizes $ s_1,\ldots, s_l $ as a spanning subgraph (where $ 2q $ of them are of odd size). Let $ H $ be the subgraph of $ G $ which is obtained from $ G[S]$ by removing all edges of these  $ l $ cliques. Therefore,

$$|E(H)|+N(f_1,\ldots, f_l) \leq |E(H)|+\sum_{i=2}^{f}|C_i|(i-1)\leq \sum_{i=2}^{f}|E(C_i,\cup_{j<i} C_j)|=|E(G)|.$$
Hence, $|E(H)|<q$. This implies that $ G[S] $ contains a complete graph of odd size as a connected component. Consequently, $ G[S] $ does not contain a perfect matching. This completes the proof.
\end{proof}
\begin{rem}
Note that the bounds in Lemma~\ref{lem:chi} are best possible, in the sense that for every integers $ n\geq 4, t\geq 1 $, there is a graph $ G $ with $ |E(G)|=\hat{g}(n,t) $ such that for every $ S\subseteq V(G) $ with $ |S|\leq 2t-1 $, we have $ \chi(G-S)\geq n-1 $. Also, there is a graph $ G $  with $ |E(G)|=g(n,t) $ such that for every $ S\subseteq V(G) $ with $ |S|\leq 2t $, if $ \chi(G-S)\leq n-2 $, then $ G[S] $ contains $ tK_2 $. To see this, let $ s_1,\ldots, s_l $ be the integers achieving $ \hat{g}(n,t) $ (resp. $ g(n,t) $) and simply let $ G $ be the disjoint union of the complete graphs $ K_{n+s_1-2}, \ldots, K_{n+s_l-2} $ (resp. $ K_{n+2s_1-2}, \ldots, K_{n+2s_l-2} $).
\end{rem}
Now, we are ready to prove \pref{thm:main}.
\begin{proof}[{\rm \textbf{Proof of \pref{thm:main}.}}]
Suppose that $ G$ is a graph such that $ \chi(G)\geq n $ and $ H $ is a graph with no isolated vertex, where $ |V(H)|\geq 2t $. Also, let $ F $ be a graph such that $ F\rightarrow (G,H) $.
First, note that $ \hat{g}(2,t)=0 $ and $ g(2,t)=t $. So, if $ n=2 $ and $ H $ contains $ tK_2 $, then $ G $ contains at least one edge and thus, $ |E(F)|\geq t=g(2,t) $. This proves the theorem for $ n=2 $.
Now, suppose that $ n\geq 3 $. We are going to prove that $  |E(F)|\geq  \hat{g}(n,t)$. For the contrary, suppose that $ |E(F)|< \hat{g}(n,t) $. Thus, by Lemma~\ref{lem:chi}(i), there exists $S \subseteq V(F) $ such that $ |S|\leq 2t-1 $ and $ \chi(F-S)\leq n-2 $. Now, color all edges whose both ends are in $ S $ by red and the other edges by blue. Since $ |S|\leq 2t-1 $, there is no red copy of $ H $ in $ F $. Therefore, since $ F\rightarrow (G,H) $, there is a blue copy of $ G $ in $ F $. Then, $ \chi(G-S)\leq \chi(F-S)\leq n-2 $. Also, since all edges of $ F $ with both ends in $ S $ are red, $S $ is a stable set in $ G $. Therefore, $ \chi(G)\leq n-1 $, which is a contradiction. This shows that $ \hat{R}(G,H)\geq \hat{g}(n,t)$. Moreover, if $ H $ contains $ tK_2 $, then the same argument using Lemma~\ref{lem:chi}(ii) implies that $ \hat{R}(G,H) \geq g(n,t) $.
\end{proof}

\section{Some consequences}
In this section, we provide some implications of \pref{thm:kng2}. First, in the following theorem, we generalize \pref{thm:erdos1} and we give a necessary and sufficient condition for \eqref{eq:Kn} being hold.
\begin{thm} \label{thm:nlarge}
For every two integers $ n\geq 2, t\geq 1 $, we have $ \hat{R}(K_n,tK_2)= \binom{n+2t-2}{2} $, if and only if either $ t^2\leq \binom{n-2}{2} $ and $ t $ is even, or $ t^2\leq \binom{n-2}{2}+1 $ and $ t $ is odd.
\end{thm}
\begin{proof}
By \pref{thm:kng2}, it is enough to prove that $ g(n,t)=\binom{n+2t-2}{2} $ if and only if $ t^2\leq \binom{n-2}{2} $, when $ t $ is even and $ t^2\leq \binom{n-2}{2}+1 $, when $ t $ is odd.
Let $ I=(s_1,\ldots, s_l) $ be a list of even integers such that $$ s_1\geq \cdots\geq s_l,~~ s_1+\cdots+s_l=2t,~~ g(n,t)= \sum_{i=1}^{l} \binom{n+s_i-2}{2}. $$ Also, suppose that $ l $ is minimum subject to these conditions. One can easily check that for every two integers $ a,b $,
\begin{equation} \label{eq:replace}
\binom{n+a-2}{2}+\binom{n+b-2}{2} \geq \binom{n+a+b-2}{2} \quad \text{if and only if} \quad ab\leq \binom{n-2}{2}.
\end{equation}
First, assume that either $ t^2\leq \binom{n-2}{2} $ and $ t $ is even, or $ t^2\leq \binom{n-2}{2}+1 $ and $ t $ is odd. We are going to prove that $ g(n,t)= \binom{n+2t-2}{2}$. If $ l\geq 2 $, then $ s_{l-1}+s_{l}\leq 2t $ and thus, $ s_{l-1}s_{l} \leq \binom{n-2}{2} $. So, by \eqref{eq:replace}, one can replace $ s_{l-1} $ and $ s_l $ in  $ I $ with $ s_{l-1}+s_l $ to obtain the new sequence $I'=(s'_1,\ldots, s'_{l-1}) $, where $$\sum_{i=1}^{l-1} \binom{n+s'_i-2}{2}\leq  \sum_{i=1}^{l} \binom{n+s_i-2}{2}=g(n,t).$$
By the definition of $g(n,t)$, we have $\sum_{i=1}^{l-1} \binom{n+s'_i-2}{2}=g(n,t)$, which is in contradiction with the minimality of $ l $. Therefore, $ l=1 $ and $ g(n,t)= \binom{n+2t-2}{2} $.\\

To see the other direction, first, suppose that $ t^2> \binom{n-2}{2} $ and $ t $ is even. Then, by \eqref{eq:replace}, we have $$ g(n,t)\leq 2\binom{n+t-2}{2}< \binom{n+2t-2}{2}. $$ Finally, suppose that  $ t^2 >\binom{n-2}{2}+1 $ and $ t $ is odd. Then, again  by \eqref{eq:replace}, we have $$ g(n,t)\leq \binom{n+t-3}{2}+\binom{n+t-1}{2}< \binom{n+2t-2}{2}. $$
This completes the proof.
\end{proof}

%
%
%

As another consequence of \pref{thm:kng2}, we are going to give a positive answer to \pref{ques:erdos}. 
Erd\H{o}s and Faudree in \cite{EF} defined
\[ \hat{r}_{_\infty} (K_n)=\lim_{t\to \infty} \frac{\hat{R}(K_n,tK_2)}{t\, \hat{R}(K_n,K_2)},
\]
and
\[M_n= \min\left\{\dfrac{\binom{n+2t-2}{2}}{t\binom{n}{2}}\mid t\in \mathbb{N}\right\},\]
and questioned if $  \hat{r}_{_\infty}(K_n)=M_n $, for every positive integer $ n $? (They showed that the equality holds for $ n=1,2,3,4 $.) By a little computation, we can see that for every integer $ n\geq 4 $,
\[M_n= \dfrac{4(2n-5)}{n(n-1)}.\]

In order to answer \pref{ques:erdos}, we need the following corollary of \pref{thm:kng2} which provides tight lower and upper bounds for $ \hat{R}(K_n,tK_2) $.
\begin{lem} \label{lem:lower}
	For every positive integers $ n\geq 3 $ and $ t\geq 1 $, we have
\[
2t(2n-5)\leq \hat{R}(K_n,tK_2)\leq 
\begin{cases}
	\lceil 2t/ (n-2)  \rceil {(n-2)(2n-5)} & n \text{ is even,} \\[2mm]
	\lceil 2t/ (n-3)  \rceil (n-3)(2n-5) &  n \text{ is odd.}
\end{cases}
\]
\end{lem}
\begin{proof}
	To prove the right inequality, note that by \pref{thm:kng2}, we have
	\begin{align*} 
	\hat{R}(K_n,tK_2)&=g(n,t)\leq \lceil t/\lfloor (n-2)/2\rfloor  \rceil \binom{n-2+2\lfloor (n-2)/2\rfloor }{2} \nonumber  \\
	&=\begin{cases}
	\lceil 2t/ (n-2)  \rceil {(n-2)(2n-5)} & n \text{ is even,} \\[2mm]
	\lceil 2t/ (n-3)  \rceil (n-3)(2n-5) &  n \text{ is odd.}
	\end{cases}
	\end{align*}
	Now, to prove the left inequality, suppose that $ I=(s_1,\ldots,s_l) $ is a list of positive integers such that $$ s_1+\cdots+s_l=t,~~ g(n,t)= \sum_{i=1}^{l} \binom{n+2s_i-2}{2}.$$
Now, note that $ |s_i-s_j|\leq 1$, for every $1\leq i,j\leq l $, since otherwise, if $ s_i\geq s_j+2  $, for some $ i,j $, one may replace $ s_i $ and $ s_j $ with $ s_i-1 $ and $ s_j+1 $ respectively and consequently reduce the summation $\sum_{i=1}^{l} \binom{n+2s_i-2}{2}$. Therefore, $ \{s_1,\ldots,s_l \}\subseteq \{q,q+1\} $, for some positive integer $ q $. Let $ r $ be the number of integers in $ I $ which are equal to $ q+1 $. Thus, $ t=lq+r $ and
	\begin{equation}\label{eq:gnt}
	g(n,t)= (l-r) \binom{n+2q-2}{2}+r\binom{n+2q}{2}.
	\end{equation}
	
	Now, for every real number $ x> 0$, let us define
	\[
	f_n(x)= \frac{1}{x} \binom{n+2x-2}{2}.
	\]
	We prove that
	\[ g(n,t)\geq t \min_{\stackrel{x\in \mathbb{Z}}{1\leq x\leq t}}f_n(x). \]
	To see this, let $ \alpha = (l-r)q/t $ and thus $ (1-\alpha) = r(q+1)/t $. Therefore, by \eqref{eq:gnt},
	\begin{align*}
	g(n,t) &= \alpha \frac{t}{q} \binom{n+2q-2}{2}+ (1-\alpha) \frac{t}{q+1} \binom{n+2q}{2} \\
	& = t (\alpha f_n(q)+ (1-\alpha) f_n(q+1)) \geq t \min_{\stackrel{x\in \mathbb{Z}}{1\leq x\leq t}}f_n(x).
	\end{align*}
	On the other hand, one may see that $ f_n'(x)= 0 $ if and only if $ 4x^2= (n-2)(n-3)$. Also, $$ n-3\leq \sqrt{(n-2)(n-3)} \leq n-2.$$ Therefore,
	\begin{align*}
	\min_{\stackrel{x\in \mathbb{Z}}{1\leq x\leq t}}f_n(x) &= \min\left\{f_n(1), f_n(t), f_n(\frac{n-3}{2}), f_n(\frac{n-2}{2})\right\}\\ & = \min\left\{\binom{n}{2}, \frac{1}{t}\binom{n+2t-2}{2}, 4n-10 \right\} = 4n-10.
	\end{align*}
	This gives the desired lower bound for $ g(n,t)= \hat{R}(K_n,tK_2) $.
\end{proof}

Now, we are ready to answer \pref{ques:erdos}.
\begin{thm}\label{thm:EFQ}
For every positive integer $ n $, we have $  \hat{r}_{_\infty}(K_n)=M_n $.
\end{thm}
\begin{proof}
The claim was proved in \cite{EF} for $ n=1,2,3,4 $. Now, assume that $ n\geq 4 $.
By \pref{lem:lower}, we have
\[ \hat{r}_{_\infty}(K_n) \leq \lim_{t\to \infty} \dfrac{(2t+n-2)(2n-5)}{t\binom{n}{2}}= \dfrac{4(2n-5)}{n(n-1)}= M_n.  \]
On the other hand, again by \pref{lem:lower}, we have
$$ \frac{\hat{R}(K_n,tK_2)}{t\, \hat{R}(K_n,K_2)}\geq \frac{t(4n-10)}{t\binom{n}{2}}= M_n.$$
Hence, $ \hat{r}_{_\infty}(K_n) \geq M_n $.
\end{proof}

\section{Concluding remarks and an open problem}
In this section, we close the paper with some supplementary remarks. First, we introduce the same problem for hypergraphs and we give an open problem in this direction. The Ramsey numbers and the size Ramsey numbers of hypergraphs are defined in a similar way as for graphs. Assume that $K_n^r$ is a complete $r$-uniform hypergraph on $n$ vertices and
$tK_r^r$ is an $r$-uniform matching of size $t$. By a straightforward argument, it turns out that
\begin{equation} \label{hypergraph}
R(K_n^r,tK_r^r)=n+(t-1)r.
\end{equation}
To see the inequality $R(K_n^r,tK_r^r)\geq n+(t-1)r$, partition $n+(t-1)r-1$ vertices into two sets $A$ and $B$ of sizes $n-r$ and $tr-1$, respectively and color all edges (all $r$-subsets) with non-empty intersection with $A$ by red and all edges which are included in $B$ by blue. Clearly this yields a 2-edge coloring of a complete $r$-uniform hypergraph on $n+(t-1)r-1$ vertices with no red copy of $K_n^r$ and no blue copy of $tK_r^r$. Now in order to show that $R(K_n^r,tK_r^r)\leq n+(t-1)r$, consider a 2-edge colored complete $r$-uniform hypergraph $H$ on $n+(t-1)r$ vertices and let $M$ be the maximum blue matching in this hypergraph. If $|M|\geq t$, then we are done. Otherwise,  $|M|\leq t-1$ and the hypergraph induced on the vertices $V(H)-V(M)$ is a monochromatic complete $r$-uniform hypergraph of color red on at least $n$ vertices and so $ H $ contains a red copy of $ K_n^r $. An argument similar to what we used in Section 2 and \eqref{hypergraph} ensure the following inequality for the size Ramsey number $\hat{R}(K_n^r,tK_r^r)$.

$$\hat{R}(K_n^r,tK_r^r)\leq g_r(n,t)=\min \left\{ \sum_{i=1}^l \binom{n+r(s_i-1)}{r} : s_i\in \mathbb{Z}^+, s_1+\cdots+s_l\geq t \right\}.$$

It was established in this paper that the equality holds for $r=2$. This gives rise to the interesting and challenging question that whether $\hat{R}(K_n^r,tK_r^r)=g_r(n,t)$ for every $r\geq 2$. We believe that our method  can be applied to attack the problem, however much efforts and  details are  definitely needed.

In the other point of view, one may think of the generalization of \pref{thm:main} to the case when there are more than two colors. 
For given graphs $G_1,\ldots, G_f$, the {\it chromatic Ramsey number}, denoted by $R_c(G_1,\ldots, G_f)$,  is the least positive integer $n$ such that there exists a graph $G$ of the chromatic number $n$ for which we have $ G\rightarrow (G_1,\ldots, G_f)$. Chromatic Ramsey numbers were introduced in 1976 by Burr, Erd\H{o}s and Lov$\acute{a}$sz \cite{BEL}. Using Lemma~\ref{lem:chi} we can find a lower bound for the size Ramsey number of graphs in terms of the chromatic Ramsey numbers. In fact, we have the following interesting inequality.
$$\hat{R}(G_1,\ldots, G_f,H_1,\ldots, H_l)\geq \hat{g}(n,t),$$
where $n=R_c(G_1,\ldots, G_f)$ and $t=\lfloor\frac{R(H_1,\ldots, H_l)}{2}\rfloor$. To see the inequality, assume to the contrary that there is a graph $G$ with $|E(G)|< \hat{g}(n,t)$ such that $G\rightarrow (G_1,\ldots, G_f,H_1,\ldots, H_l)$. Using Lemma~\ref{lem:chi}(i), there exists a subset $S\subseteq V(G) $ such that $|S|\leq 2t-1$ and $\chi(G-S)\leq n-2 $. Suppose that $G'$ is a graph obtained from $G$ by removing  all  edges  contained in $S$. Clearly, $\chi(G')\leq \chi(G-S)+1\leq n-1$ and so there is an $f$-coloring of the edges of $G'$ by colors $1,2,\ldots,f$ such that $ G' $ contains no copy of $G_i$ of color $ i $, for every $1\leq i\leq f$. On the other hand, $|S|\leq 2t-1$ and $2t\leq R(H_1,\ldots, H_l)$, so there is an $l$-coloring of the edges within $S$ by colors $f+1,\ldots,f+l$ such that $ G $ contains no copy of $H_i$ of color $ f+i $ for each $1\leq i\leq l$. Combining these two edge colorings yields an edge coloring for $G$ with $f+l$ colors $1,\ldots,f+l$, where the subgraph
induced by the edges of color $i$ (resp. $f+i$) does not contain $G_i$ (resp. $H_i$) as a subgraph for each $1\leq i\leq f$ (resp. $1\leq i\leq l$). This observation contradicts the fact $G\rightarrow (G_1,\ldots, G_f,H_1,\ldots, H_l)$. With a similar argument we have
$$\hat{R}(G_1,\ldots, G_f,H_1,\ldots, H_l)\geq g(n,t),$$
where $n=R_c(G_1,\ldots, G_f)\geq 4$ and $t=\lfloor\frac{R(H_1,\ldots, H_l)-1}{2}\rfloor$.

\end{document}